\documentclass[11pt, reqno]{amsart}

\usepackage{amsmath, amsfonts, amssymb, amsbsy, bigstrut, graphicx, enumerate,  upref}

\usepackage{pstricks}

\newcommand{\mk}{\mathcal{K}}

\newtheorem{thm}{Theorem}[section]
\newtheorem{lmm}[thm]{Lemma}
\newtheorem{cor}[thm]{Corollary}

\newcommand{\ee}{\mathbb{E}}

\newcommand{\pp}{\mathbb{P}}

\newcommand{\ra}{\rightarrow}
\newcommand{\rr}{\mathbb{R}}

\newcommand{\var}{\mathrm{Var}}
\newcommand{\ve}{\varepsilon}

\newcommand{\bh}{\hat{\beta}}
\newcommand{\hk}{\hat{K}}
\newcommand{\hm}{\hat{\mu}'}
\newcommand{\hmu}{\hat{\mu}}  

\newcommand{\ck}{\mathcal{K}}

\numberwithin{equation}{section}

\usepackage{hyperref}

\begin{document}
\title{Prediction error of cross-validated Lasso}
\author{Sourav Chatterjee}
\address{\newline Department of Statistics \newline Stanford University\newline Sequoia Hall, 390 Serra Mall \newline Stanford, CA 94305\newline \newline \textup{\tt souravc@stanford.edu}}
\author{Jafar Jafarov}
\address{\newline Department of Mathematics \newline Stanford University\newline 450 Serra Mall, Building 380 \newline Stanford, CA 94305\newline \newline \textup{\tt jafarov@stanford.edu}}
\thanks{Research partially supported by NSF grant DMS-1441513}

\begin{abstract}
In spite of the wealth of literature on the theoretical properties of the Lasso, there is very little known when the value of the tuning parameter is chosen using the data, even though this is what actually happens in practice. We give a general upper bound on the prediction error of Lasso when the tuning parameter is chosen using a variant of 2-fold cross-validation. No special assumption is made about the structure of the design matrix, and the tuning parameter is allowed to be optimized over an arbitrary data-dependent set of values. The proof is based on  a general principle that may extend to other kinds of cross-validation as well as to other penalized regression methods. Based on this result, we propose a new estimate for error variance  in high dimensional regression and prove that it has good properties under minimal assumptions. 
\end{abstract}

\maketitle

\section{Introduction}
Since its introduction by Tibshirani \cite{tibs96}, the Lasso has become one of the most popular tools for high dimensional regression. Although most readers of this article will undoubtedly be familiar with the Lasso, let us still describe the setup for the sake of fixing notation. Consider the linear regression model 
\begin{equation}\label{lin1}
Y = X\beta + \ve\, ,
\end{equation}
where $X$ is an $n\times p$ design matrix, $\beta$ is a $p\times 1$ vector of unknown parameters, $\ve$ is a $n\times 1$ vector of i.i.d.~$N(0,\sigma^2)$ random variables (where $\sigma^2$ is an unknown parameter called the `error variance'), and $Y$ is the $n\times 1$ response vector. When $p$ is small and $n$ is large, ordinary least squares is an effective tool for estimating the parameter vector $\beta$. However, many modern applications have the characteristic that both $n$ and $p$ are large, and sometimes $p$ is much larger than $n$. The Lasso prescribes a way of estimating $\beta$ in this scenario. There are two equivalent versions of the Lasso, namely, the primal version and the dual version. In the primal version, the statistician chooses a tuning parameter $K$, and produces an estimate $\bh$ by minimizing $\|Y-X\beta\|$ subject to $|\beta|_1\le K$, where $\|\cdot\|$ is the Euclidean norm in $\rr^n$ and $|\cdot|_1$ is the $\ell^1$ norm in $\rr^p$.  In the dual version, the statistician chooses a tuning parameter $\lambda$, and produces an estimate $\bh$ by minimizing $\|Y-X\beta\|^2 + \lambda |\beta|_1$. Note that the optimal $\bh$ may not be unique in either version. Although the Lasso was introduced in its primal form in Tibshirani's paper~\cite{tibs96}, the dual form has become more popular due to algorithmic efficiency. There is no universally accepted prescription for choosing the tuning parameter in either of the two versions. In practice, the tuning parameter is almost always chosen in a data-dependent manner, often using cross-validation.

There is a large body of literature on the theoretical properties of the Lasso. Instead of trying to give a comprehensive overview, we will just  highlight some essential references from this literature. The analysis of basis pursuit by Chen, Donoho and Saunders~\cite{chenetal98} and the papers of Donoho and Stark \cite{donohostark} and Donoho and Huo \cite{donohohuo01} provided some key ideas for subsequent authors. Knight and Fu \cite{knightfu00} proved consistency of $\bh$ under the assumption that $p$ remains fixed and $n\ra\infty$. When both $p$ and $n$ tend to infinity, the consistency of $\bh$ has no traditional definition. Greenshtein and Ritov~\cite{gr} defined a notion of consistency in this setting that they called `persistence', and proved that under a set of assumptions, the Lasso estimator is persistent. Persistence of a sequence of estimators is defined as follows. Rewrite the linear regression model \eqref{lin1} as
\begin{equation}\label{lin2}
y_i = x_i \beta + \ve_i\, , \ \ i=1,\ldots, n\, ,
\end{equation}
where $x_1,\ldots, x_n$ are $n$ rows of the design matrix $X$, $\ve_1\ldots, \ve_n$ are the components of the error vector $\ve$, and $y_1,\ldots,y_n$ are the components of the response vector $Y$. Suppose  that the pairs $(y_i,x_i)$ are i.i.d.~draws from some probability distribution $F_n$ on $\rr\times \rr^{p}$. For any estimator $\bh$, define
\begin{equation}\label{lndef}
L_n(\bh) := \ee(y-x\bh)^2\, ,
\end{equation}
where $(y,x)$ is a pair drawn from $F_n$ that is independent of $(y_1,x_1),\ldots, (y_n,x_n)$. Similarly for any $\beta\in \rr^p$ define $L_n(\beta) := \ee(y-x\beta)^2$. If we have a sequence of regression problems as above and estimators $\bh_n$, Greenshtein and Ritov \cite{gr} called the sequence $\bh_n$ `persistent' if 
\[
\lim_{n\ra\infty}(L_n(\bh_n) - \inf_\beta L_n(\beta) )=0\, .
\]
Persistence has become a popular notion of consistency in high dimensional problems. It is sometimes called `risk consistency'. It has been  the topic of investigation in several subsequent papers on the Lasso, such as Bunea et al.~\cite{btw06, btw07} and van de Geer \cite{vdg08}. Quantitative bounds were given in Zhang~\cite{zhang09}, Rigollet and Tsybakov \cite{rigollet}, B\"uhlmann and van de Geer \cite{bg11}, Bartlett et al.~\cite{bartlettetal12} and Chatterjee \cite{chatterjee14, chatterjee14b}. 

Another kind of consistency that has been investigated in the context of Lasso is model selection consistency. The Lasso estimator has the property that often, most of the coordinates of $\bh$ turn out to be equal to zero. The nonzero coordinates therefore do an automatic `model selection'. Consistency of model selection by Lasso under a variety of assumptions on the design matrix and the sparsity of $\beta$ was investigated by Zou~\cite{zou06}, Donoho et al.~\cite{donohoetal06},  Wainwright~\cite{wainwright09},  Meinshausen and B\"uhlmann \cite{meinshausenbuhlmann06}, Meinshausen and Yu \cite{meinshausenyu09}, Bickel et al.~\cite{bickeletal09}, Massart and Meynet \cite{massartmeynet11} and Zhao and Yu \cite{zhaoyu07}. 

In all of the above papers, the tuning parameter is considered to be deterministically chosen. For example, Greenshtein and Ritov \cite{gr} showed that if the tuning parameter $K$ in the primal form of the Lasso grows like $o((n/\log n)^{1/4})$, then the Lasso estimator is persistent. This is the general flavor of subsequent results, such as those in \cite{bg11, btw07, donohoetal06, meinshausenbuhlmann06, meinshausenyu09, vdg08, wainwright09, zhaoyu07, zou06}. 

However, as mentioned before, this is usually not how the tuning parameter is chosen in practice. Systematic ways of deterministically choosing the tuning parameter have been proposed by Wang and Leng \cite{wangleng07}, Zou et al.~\cite{zouetal07} and Tibshirani and Taylor \cite{tt}. Other authors, such as Tibshirani~\cite{tibs96, tibs11}, Greenshtein and Ritov~\cite{gr}, Hastie et al.~\cite{hastieetal09}, Efron et al.~\cite{efronetal04}, Zou et al.~\cite{zouetal07}, van de Geer and Lederer \cite{vdgl13}, Fan et al.~\cite{fanetal12} and Friedman \cite{friedmanetal10} recommend using cross-validation to select the value of the tuning parameter. In practice, the tuning parameter is almost always chosen using some data-dependent method, often cross-validation.

In view of the above, it is surprising that there are very few rigorous results about the Lasso when the tuning parameter is chosen through cross-validation. In fact, the only results we are aware of are from some recent papers of Lecu\'e and Mitchell \cite{lecuemitchell12} and Homrighausen and McDonald~\cite{hm13a, hm13, hm14}. The paper \cite{lecuemitchell12} aims to build a general theory of cross-validation in a variety of problems; unfortunately, it seems that for cross-validation in Lasso, it requires that the vector $x$ of explanatory variables is a random vector with a log-concave distribution. This is possibly too strong an assumption to be practically useful. The papers \cite{hm13a, hm13, hm14} have more relaxed assumptions. Roughly speaking, the main result of \cite{hm13} goes as follows. Consider the primal form of Lasso, and suppose that the tuning parameter $K$ is chosen from an interval $[0,K_n]$ using $r$-fold cross-validation, where $r$ is a fixed number and $K_n$ is determined according to a formula given in \cite{hm13}. Let $\hk$ be the optimized value of $K$, and let $\bh_{\hk}$ be the Lasso estimate of $\beta$ for this value of the tuning parameter. Suppose that $(y_i,x_i)$ are i.i.d.~from some distribution $F_n$, and let $L_n(\bh_{\hk})$ be defined as in \eqref{lndef}. Additionally suppose that $p$ is growing like $n^\alpha$ for some positive $\alpha$, and that the minimum nonzero singular value of the design matrix $X$ is, with high probability, of order $n^{\alpha/2}$ or bigger. Under these conditions and a few other assumptions, the main result of \cite{hm13} states that, as $n\ra\infty$, 
\[
L_n(\bh_{\hk}) - \inf_{\beta\, :\, |\beta|_1\le K_n} L_n(\beta) = o\biggl(K_n^2\sqrt{\frac{\log n}{n}}\biggr)\, .
\]
Although this is laudable as the first mathematical result about any kind of  consistency of cross-validated lasso, there are a number of unsatisfactory aspects of this result. For the theorem to be effective, we need 
\begin{equation}\label{kncond}
K_n = o(n^{1/4}(\log n)^{-1/4})\, ,
\end{equation}
which is too slow for all practical purposes. Usually, the tuning parameter is optimized over a fairly large range, determined by the statistician using some ad hoc data-dependent rule. The second issue is that the value of $K_n$ is prescribed by a formula given by the authors, which may not end up satisfying \eqref{kncond}. Lastly, the condition on the smallest nonzero singular value of the design matrix looks a bit restrictive, especially since it has been observed in several recent papers that risk consistency in Lasso (with deterministic value of the tuning parameter) holds without any conditions on the design matrix \cite{ bartlettetal12, bg11, chatterjee14, rigollet}. 

The goal of this paper is to address these issues and prove a new and better upper bound on the prediction error of cross-validated Lasso under fewer assumptions. The main result is presented in the next section. An application to error variance estimation is worked out in Section \ref{evsec}.

Incidentally, there is a substantial body of literature on cross-validation for ridge regression, the classical cousin of Lasso. A representative paper from this literature, for example, is the highly cited article of Golub, Heath and Wahba~\cite{golubheathwahba79}. The techniques and results of these papers depend heavily on the friendly mathematical structure of ridge regression. They do not seem to generalize to other settings in any obvious way. 

Similarly, classical results on cross-validation such as those of Stone~\cite{stone74, stone77}, apply only to problems where $p$ is fixed, and are therefore not relevant in our setting. 

It is important to point out that the Lasso is not the only technique for high dimensional regression under sparsity assumptions. Numerous methods have been proposed in the last twenty years. Many of them, like the Lasso, are based on the idea of performing regression with a penalty term. These include basis pursuit~\cite{chenetal98}, SCAD~\cite{fanli01}, LARS~\cite{efronetal04}, elastic net~\cite{zh05} and the Dantzig selector~\cite{candestao07}. The Lasso itself has been sophisticated over the years, yielding variants such the group Lasso~\cite{yuanlin06}, the adaptive Lasso~\cite{zou06} and the square-root Lasso~\cite{bcw11}. Penalized regression is not the only approach; for example, methods of model selection by testing hypotheses have been proposed in \cite{abramovichetal06,barbercandes15, birgemassart01,birge06}. Most of these methods involve some sort of a tuning parameter, which is often chosen using cross-validation. The techniques of this paper may be helpful in analyzing cross-validation in a variety of such instances. Our reason for focusing on the Lasso is simply to choose one test case where the proof technique may be implemented, and the Lasso seemed like a natural choice because of its popularity among practitioners.

\section{Main result}\label{mainsec}
Consider the linear regression model \eqref{lin1}. Let $x_i$ be the $i$th row of $X$, $y_i$ be the $i$th component of $Y$ and $\ve_i$ be the $i$th component of $\ve$, so that \eqref{lin2} holds. In our setup (unlike \cite{gr}), the row vectors $x_1,\ldots, x_n$ are non-random elements of $\rr^p$. 

The mean squared prediction error of an estimator $\bh$ is defined as
\begin{equation*}\label{mspe1}
\textup{MSPE}(\bh) := \ee\biggl(\frac{\|X\beta^*-X\bh\|^2}{n}\biggr)\, ,
\end{equation*}
where $\beta^*$ is the true value of $\beta$. 

Our goal is to produce an estimate $\bh^{\textup{CV}}$ of $\beta$  using the Lasso procedure and optimizing the tuning parameter by a certain variant of $2$-fold cross-validation, and then give an upper bound for its mean squared prediction error. The specific algorithm that we are proposing is the following.
\begin{enumerate}[1.]
\item Divide the set $\{1,\ldots, n\}$ randomly into two parts $I$ and $I^c$, by independently putting each element into either $I$ or $I^c$ with equal probability. 
\item For each $K\ge 0$, let $\bh^{(K,1)}$ be a minimizer of 
\[
\sum_{i\in I} (y_i-x_i\beta)^2
\]
subject to $|\beta|_1\le K$ and let $\bh^{(K,2)}$ be a minimizer of 
\[
\sum_{i\in I^c} (y_i-x_i\beta)^2
\]
subject to $|\beta|_1\le K$. If there are multiple minima, choose one according to some deterministic rule. In the rare event that $I$ or $I^c$ is empty, define the corresponding $\bh$'s to be zero. 
\item Let $N_1$ and $N_2$ be two nonnegative integer-valued random variables, where $N_1$ is a function of $(y_i,x_i)_{i\in I^c}$ and $N_2$ is a function of $(y_i,x_i)_{i\in I}$. These numbers will determine the range over which the tuning parameter is optimized in the next step. The choice of $N_1$ and $N_2$ is left to the user. 
\item Let $\delta$ be a positive real number, to be chosen by the user. Let $\hk_1$ be a minimizer of 
\[
\sum_{i\in I} (y_i - x_i \bh^{(K,2)})^2
\]
as $K$ ranges over the set $\{0,\delta, 2\delta,\ldots,N_1\delta\}$. Let $\hk_2$ be a minimizer of 
\[
\sum_{i\in I^c} (y_i - x_i \bh^{(K,1)})^2
\]
as $K$ ranges over the set $\{0,\delta,2\delta,\ldots,N_2\delta\}$. 
\item Traditional cross-validation produces a single optimized $\hk$ instead of $\hk_1$ and $\hk_2$ as we did above. In this step, we will combine $\hk_1$ and $\hk_2$ to produce a single $\hk$, as follows. Define a vector $\hm\in \rr^n$ as 
\[
\hm_i :=
\begin{cases}
x_i \bh^{(\hk_1,2)} &\text{ if } i\in I\, ,\\
x_i \bh^{(\hk_2, 1)} &\text{ if } i\in I^c\, .
\end{cases}
\]
For each $K$, let $\bh^{(K)}$ be a minimizer of $\|Y-X\beta\|$ subject to $|\beta|_1\le K$. Let $\hk$ be a minimizer of $\|\hm - X\bh^{(K)}\|$ over $K\ge 0$.
\item Finally, define the cross-validated estimate $\bh^{\textup{CV}} := \bh^{(\hk)}$. 
\end{enumerate}
The following theorem gives an upper bound on the mean squared prediction error of $\bh^{\textup{CV}}$ under the condition that $N_1\delta$ and $N_2\delta$ both exceed $|\beta^*|_1$. This is the main result of this paper. In the statement of the theorem, we use the standard convention that for a random variable $X$ and an event~$A$, $\ee(X;A)$ denotes the expectation of the random variable $X1_A$, where~$1_A=1$ if $A$ happens and $0$ otherwise. 
\begin{thm}\label{mainthm}
Consider the linear regression model \eqref{lin1}. Let $\bh^{\textup{CV}}$ be defined as above and $\beta^*$ be the true value of $\beta$. Let $L := |\beta^*|_1+\delta$, and let
\begin{align*}
&M := \max_{1\le j\le p}\frac{1}{n}\sum_{i=1}^n x_{ij}^4\,, \ \ l_1 := \ee\log (N_1+1)\,,\ \ l_2 := \ee\log (N_2+1)\, .
\end{align*}
Then
\begin{align*}
\ee\biggl(\frac{\|X\beta^*-X\bh^{\textup{CV}}\|^2}{n};\, N_1\delta \ge |\beta^*|_1,\, N_2\delta\ge |\beta^*|_1\biggr) &\leq  C_1 \frac{\sqrt{l_1}+\sqrt{l_2}}{\sqrt{n}}+ C_2 \sqrt{\frac{\log (2p)}{n}} + E_n\, ,
\end{align*}
where
\begin{align*}
&C_1 = 16(4\sigma^4 +2L^2M^{1/2}\sigma^2)^{1/2}\,, \\
& C_2 = 96L^2M^{1/2} + 57LM^{1/4}\sigma \, ,
\end{align*}
and $E_n$ is the exponentially small term
\[
16\biggl(\frac{(n+5)\sigma^4}{n} + \frac{(n+1)\sigma^2}{n}L^2M^{1/2}\biggr)^{1/2}\biggl(\frac{1+2^{-1/2}}{2}\biggr)^{n/2}\,.
\]
\end{thm}
\noindent{\it \underline{Remarks}}
\begin{enumerate}[1.]
\item To understand what the theorem is saying, think of $L$, $M$ and $\sigma$ as fixed numbers that are not growing with $n$. Also, think of $\delta$ as tending to zero (or at least remaining bounded), so that $L$ is basically the same as $|\beta^*|_1$. Then the theorem says that as long as $N_1\delta$, $N_2\delta$ and $p$ tend to infinity slower than exponentially with $n$, the mean squared prediction error tends to zero as $n\ra\infty$. In fact, the prediction error goes to zero even if $|\beta^*|_1$ is allowed to grow, as long as it grows slower than $N_1\delta$, $N_2\delta$, and $n^{1/4}(\log p)^{-1/4}$. The last criterion is a familiar occurrence in papers on the consistency of Lasso, as mentioned before.
\item There are five significant advances that Theorem \ref{mainthm} makes over existing results on cross-validated Lasso~\cite{hm13, hm14}:
\begin{enumerate}[(i)]
\item The range of values over which the tuning parameter is optimized is allowed to grow exponentially in $n$.
\item The range of optimization is allowed to be arbitrarily data-dependent. 
\item The error bound depends on the $\ell^1$ norm of the true $\beta$,  and not on the range of values over which the tuning parameter is optimized (except through a logarithmic factor).
\item The theorem imposes essentially no condition on the design matrix.
\item The theorem gives a concrete error bound on the prediction error instead of an asymptotic persistence proof.
\end{enumerate}
\item Theorem \ref{mainthm} is silent on how to choose $N_1$, $N_2$ and $\delta$; the choice is left to the practitioner. Clearly, for the upper bound to be meaningful, it is necessary that with high probability  both $N_1\delta$ and $N_2\delta$ exceed $|\beta^*|_1$. This is not surprising, because if the range of values over which the tuning parameter is optimized does not contain the $\ell^1$ norm of the true $\beta$, then primal Lasso is unlikely to perform well. It is not clear whether one can produce a choice of the range that has a theoretical guarantee of success. 
\item It would be nice to have a similar result for the versions of cross-validation that are actually used in practice. The main reason why we work with our version is that it is mathematically tractable. It is possible that a variant of the techniques developed in this paper, together with some new ideas, may lead to a definitive result about traditional cross-validation some day.
\end{enumerate}
The next section contains an application of Theorem \ref{mainthm} to error variance estimation. Theorem~\ref{mainthm} is proved in Section~\ref{proofsec}, and the results of Section \ref{evsec} are proved in Section \ref{corsec}.


\section{Application to error variance estimation}\label{evsec}
Error variance estimation is the problem of estimating $\sigma^2$ in the linear regression model \eqref{lin1}. In the high dimensional case, this problem has gained some prominence in recent times, partly because of the emergence of literature on significance tests for Lasso \cite{javanmardmontanari14, lockhartetal13}. Significance tests almost always require some plug-in estimate of $\sigma^2$. There are a number of proposed methods for error variance estimation in Lasso~\cite{dicker14, fanetal12, stadleretal10, sunzhang10, sunzhang12, sunzhang13}. The recent paper of Reid, Tibshirani and Friedman~\cite{reidetal13} gives a comprehensive survey of these techniques and compares their strengths and weaknesses through extensive simulation studies. They summarize their findings as follows (italics and quotation marks added):
\begin{quote}
{\it ``Despite some comforting asymptotic results, finite sample performance of these estimators seems to suffer, particularly when signals become large and non sparse. Variance estimators based on residual sums of squares with adaptively chosen regularization parameters seem to have promising finite sample properties. In particular, we recommend the cross-validation based, Lasso residual sum of squares estimator as a good variance estimator under a broad range of sparsity and signal strength assumptions. The complexity of their structure seems to have discouraged their rigorous analysis.''}
\end{quote}
Reid, Tibshirani and Friedman \cite{reidetal13} observe that it is possible to construct an error variance estimator using the cross-validation procedure of Homrighausen and McDonald \cite{hm13}, but it is consistent only when $\hat{s}/n \ra 0$, where $\hat{s}$ is the number of nonzero entries in the cross-validated Lasso estimate~$\bh$. Such a result does not follow from any known theorem in the literature.

In an attempt to address the above problems, we propose a new estimate of error variance based on our cross-validated Lasso estimate. Let $\bh^{\textup{CV}}$ be the cross-validated estimate of $\beta$ defined in Section \ref{mainsec}. Define
\[
\hat{\sigma}^2 := \frac{\|Y-X\bh^{\textup{CV}}\|^2}{n}\, .
\]
The following theorem gives an upper bound on the mean absolute error of $\hat{\sigma}^2$. It gives a theoretical proof that $\hat{\sigma}^2$ is a good estimator of $\sigma^2$ under the same mild conditions as in Theorem \ref{mainthm}. Whether it will actually perform well in practice is a different question that is not addressed in this paper. 
\begin{thm}\label{sigma}
Let all notation be as in Theorem \ref{mainthm} and let $\hat{\sigma}^2$ be defined as above. Let $R$ denote the right-hand side of the inequality in the statement of Theorem \ref{mainthm}. Then
\[
\ee(|\hat{\sigma}^2-\sigma^2|;\, N_1\delta\ge |\beta^*|_1,\, N_2\delta\ge |\beta^*|_1) \le \sigma^2 \sqrt{\frac{2}{n}} +2\sigma \sqrt{R} + R\, .
\]
\end{thm}
The following corollary demonstrates a simple scenario under which  $\hat{\sigma}^2$ is a consistent estimate of $\sigma^2$. Note that Theorem \ref{sigma} is more general than this illustrative corollary. 
\begin{cor}\label{sigmacor}
Consider a sequence of regression problems and estimates of the type analyzed in Theorems \ref{mainthm} and \ref{sigma}. Suppose that $\sigma^2$ is the same in each problem, but $n\ra\infty$ and all other quantities change with $n$. Assume that:
\begin{enumerate}[(i)]
 \item The entries of the design matrix and the $\ell^1$ norm of the true parameter vector $\beta^*$ remain uniformly bounded as $n\ra\infty$. 
 \item $N_1\delta$ and $N_2\delta$ tend to infinity in probability, but $\delta$ remains bounded.
 \item $\log N_1$, $\log N_2$ and $\log p$ grow at most like $o(n)$. 
\end{enumerate}
Then $\hat{\sigma}^2$ is a consistent estimate of $\sigma^2$ in this sequence of problems. 
\end{cor}
Theorem \ref{sigma} and Corollary \ref{sigmacor} are proved in Section \ref{corsec}. In the next section, we prove Theorem~\ref{mainthm}.

\section{Proof of Theorem \ref{mainthm}}\label{proofsec}
In the statement of the theorem, $L$ is defined as $|\beta^*|_1+\delta$. However in this proof, we will redefine $L$ as the smallest integer multiple of $\delta$ that is $\ge |\beta^*|_1$. It suffices to prove the theorem with this new $L$, because the old $L$ is $\ge$ the new one.

Let $\mu^* := X\beta^*$ and $\hmu := X\bh^{\textup{CV}}$. We will first work on bounding $\|\hm-\mu^*\|$ instead of $\|\hmu-\mu^*\|$. If $I^c$ is nonempty and $N_1\delta \ge |\beta^*|_1$, then by definition of $\hat{K}_1$,
$$\sum_{i\in I}(y_i-x_i\hat{\beta}^{(\hat{K}_1,2)})^2 \leq \sum_{i\in I}(y_i-x_i\hat{\beta}^{(L,2)})^2\,.$$
If $I^c$ or $I$ is empty, then equality holds, so the above inequality is true anyway. Adding and subtracting $x_i\beta^*$ inside the square on both sides gives
\begin{align*}
&\sum_{i\in I} (\ve_i^2 +2\ve_i(x_i\beta^* - x_i\hat{\beta}^{(\hat{K}_1,2)}) + (x_i\beta^* - x_i\hat{\beta}^{(\hat{K}_1,2)})^2) \\
&\le \sum_{i\in I} (\ve_i^2 +2\ve_i(x_i\beta^* - x_i\hat{\beta}^{(L,2)}) + (x_i\beta^* - x_i\hat{\beta}^{(L,2)})^2) \, .
\end{align*}
This can be rewritten as 
\begin{align}
\sum_{i\in I}(\mu_i^*-\hm_i)^2&=\sum_{i\in I}(x_i\beta^*-x_i\hat{\beta}^{(\hat{K}_1,2)})^2\notag\\
&\leq 2\sum_{i\in I}\varepsilon_i(x_i\hat{\beta}^{(\hat{K}_1,2)}-x_i\hat{\beta}^{(L,2)})+\sum_{i\in I}(x_i\beta^*-x_i\hat{\beta}^{(L,2)})^2\, .\label{mainineq}
\end{align}
A similar expression may be  obtained for $\sum_{i\in I^c}(\mu_i^*-\hm_i)^2$. Since these two quantities have the same unconditional distribution and their sum is the total prediction error $\|\mu^*-\hm\|^2$, it suffices to obtain a bound on the expectation of one of them. We start by bounding expectation of the second term in \eqref{mainineq}. Throughout the remainder of the proof, let $\ee'$ denote conditional expectation given $I$ and $\ee''$ denote the conditional expectation given $I$ and $(y_i,x_i)_{i\in I^c}$. 
\begin{lmm}\label{l1}
For each $\beta \in \rr^p$, let
\[
\varphi(\beta):=\sum_{i\in I}(x_i\beta^*-x_i\beta)^2-\sum_{i\in I^c}(x_i\beta^*-x_i\beta)^2\, .
\]
Then
\begin{align*}
\ee\biggl(\sup_{\beta\in \mathbb{R}^p\,:\,|\beta|_1\leq L}\varphi(\beta)\biggr) \leq 3L^2(2Mn\log(2p^2))^{1/2}\,.
\end{align*}
\end{lmm}
\begin{proof}
Note that we can write
$$\varphi(\beta)=\sum_{i=1}^{n}\eta_i(x_i\beta^*-x_i\beta)^2\,,$$
where
$$\eta_i=\begin{cases}1 & \textnormal{if }i\in I,\\
-1&\textnormal{if }i \in I^c.
\end{cases}$$
Note that $\eta_i$ are i.i.d.~random variables with mean zero. Write 
$$\sum_{i=1}^{n}\eta_i(x_i\beta^*-x_i\beta)^2=\sum_{i=1}^{n}\eta_i(x_i\beta^*)^2-2\sum_{i=1}^{n}\eta_i(x_i\beta^*)(x_i\beta)+\sum_{i=1}^{n}\eta_i(x_i\beta)^2\, .$$
First, observe that 
\begin{align}\label{step1}
\ee\biggl(\sum_{i=1}^{n}\eta_i(x_i\beta^*)^2\biggr)=0\, .
\end{align}
Next, note that for any $\beta$ with $|\beta|_1\le L$,  
\begin{align}
-2\sum_{i=1}^{n}\eta_i(x_i\beta^*)(x_i\beta)&=-2\sum_{j=1}^{p}\beta_j\biggl(\sum_{i=1}^n\eta_i(x_i\beta^*)x_{ij}\biggr)\nonumber\\
&\leq 2L\max_{1\leq j\leq p}\biggl|\sum_{i=1}^n\eta_i(x_i\beta^*)x_{ij}\biggr|\,.\label{eta0}
\end{align}
Lemma \ref{app1} from the Appendix implies that for any $\theta\in \rr$, 
$$\mathbb{E}\biggl(\exp\biggl(\theta\sum_{i=1}^n\eta_i(x_i\beta^*)x_{ij}\biggr)\biggr) \leq \exp\biggl(\frac{\theta^2}{2}\sum_{i=1}^n\bigl((x_i\beta^*)x_{ij}\bigr)^2\biggr)\,.$$
So by Lemma \ref{app2} of the Appendix,
\begin{align}\label{eta1}
\mathbb{E}\biggl(\max_{1\leq j\leq p}\biggl|\sum_{i=1}^n\eta_i(x_i\beta^*)x_{ij}\biggr|\biggr) &\leq \biggl(2\log(2 p)\max_{1\leq j \leq p}\sum_{i=1}^n\bigl((x_i\beta^*)x_{ij}\bigr)^2\biggr)^{1/2}\,.
\end{align}
Note that for any $j$,
\begin{align*}
\sum_{i=1}^n\bigl((x_i\beta^*)x_{ij}\bigr)^2&= \sum_{i=1}^n \biggl(\sum_{k=1}^px_{ik}x_{ij} \beta_k^*\biggr)^2 \\
&= \sum_{i=1}^n\sum_{1\le k,l\le p}x_{ik} x_{il} x_{ij}^2 \beta^*_k\beta^*_l\\
&\le |\beta^*|_1^2 \max_{1\le k,l\le p} \biggl|\sum_{i=1}^nx_{ik} x_{il} x_{ij}^2\biggr|\le n M|\beta^*|_1^2\, .
\end{align*}
Therefore by \eqref{eta1}, 
\begin{align*}
\mathbb{E}\biggl(\max_{1\leq j\leq p}\biggl|\sum_{i=1}^n\eta_i(x_i\beta^*)x_{ij}\biggr|\biggr) &\leq (2Mn|\beta^*|_1^2\log(2 p))^{1/2}\, .
\end{align*}
Using this information in \eqref{eta0}, we get
\begin{align}\label{step2}
\ee\biggl(\sup_{\beta\in \rr^p\, :\, |\beta|_1\le L}-2\sum_{i=1}^{n}\eta_i(x_i\beta^*)(x_i\beta)\biggr)&\le 2L^2(2Mn\log (2p))^{1/2}\, .
\end{align}
Similarly, for any $\beta$ with $|\beta|_1\le L$, 
\begin{align}
\sum_{i=1}^{n}\eta_i(x_i\beta)^2&=\sum_{j,k=1}^p\beta_j\beta_k\biggl(\sum_{i=1}^n\eta_ix_{ij}x_{ik}\biggr)\nonumber\\
&\leq L^2\max_{1\leq j, k\leq p}\biggl|\sum_{i=1}^n\eta_ix_{ij}x_{ik}\biggr|\, .\label{eta2}
\end{align}
By Lemma \ref{app1} from the Appendix, we have that for any $\theta\in \rr$, 
$$\mathbb{E}\biggl(\exp\biggl(\theta\sum_{i=1}^n\eta_ix_{ij}x_{ik}\biggr) \biggr)\leq \exp\biggl(\frac{\theta^2}{2}\sum_{i=1}^n\bigl(x_{ij}x_{ik}\bigr)^2\biggr)\, .$$ 
Therefore by Lemma \ref{app2} of the Appendix, 
$$\mathbb{E}\biggl(\max_{1\leq j,k\leq p}\biggl|\sum_{i=1}^n\eta_ix_{ij}x_{ik}\biggr|\biggr) \leq \biggl(2\log (2p^2)\max_{1\leq j,k \leq p}\sum_{i=1}^n\bigl(x_{ij}x_{ik}\bigr)^2\biggr)^{1/2}\le (2Mn\log (2p^2))^{1/2}\,.$$
Therefore by \eqref{eta2},
\begin{align}\label{step3}
\mathbb{E}\biggl(\sup_{\beta\in \rr^p\, :\, |\beta|_1\le L}\sum_{i=1}^{n}\eta_i(x_i\beta)^2\biggr) \leq L^2(2Mn\log(2p^2))^{1/2}\,.
\end{align}
By combining \eqref{step1}, \eqref{step2} and \eqref{step3} we get the desired result.
\end{proof}

\begin{lmm}\label{l2}
\[
\mathbb{E}\biggl(\sum_{i\in I^c}(x_i\beta^*-x_i\hat{\beta}^{(L,2)})^2\biggr) \leq 2L\sigma (2M^{1/2}n\log(2p))^{1/2}\, .
\]
\end{lmm}
\begin{proof}
The inequality is trivially true if $I^c$ is empty. So let us assume that $I^c$ is nonempty. By definition of $\bh^{(L,2)}$, 
\[
\sum_{i\in I^c}(y_i - x_i \bh^{(L,2)})^2 \le \sum_{i\in I^c}(y_i - x_i \beta^*)^2\, .
\]
Adding and subtracting $x_i\beta^*$ inside the bracket on the left, this becomes
\[
\sum_{i\in I^c}(\ve_i^2 +2\ve_i(x_i\beta^* - x_i \bh^{(L,2)}) + (x_i\beta^* - x_i \bh^{(L,2)})^2) \le \sum_{i\in I^c}\ve_i^2\,,
\]
which is the same as
\[
\sum_{i\in I^c}(x_i\beta^* - x_i \bh^{(L,2)})^2\le 2\sum_{i\in I^c} \ve_i(x_i\bh^{(L,2)} - x_i\beta^*)\, .
\]
Since $\ee'(\ve_i) = 0$ for all $i$ and $|\beta^{(L,2)}|_1\le L$, this gives
\begin{equation}\label{o1}
\ee'\biggl(\sum_{i\in I^c}(x_i\beta^* - x_i \bh^{(L,2)})^2\biggr) \le 2\,\ee'\biggl(\sup_{\beta\in \rr^p\, :\, |\beta|_1\le L}\sum_{i\in I^c} \ve_i x_i\beta\biggr)\, .
\end{equation}
For any $\beta$ such that $|\beta|_1\le L$, 
\begin{align}
\sum_{i\in I^c} \ve_i x_i\beta &= \sum_{i\in I^c} \sum_{j=1}^p \ve_i x_{ij} \beta_j \nonumber\\
&\le L\max_{1\le j\le p} \biggl|\sum_{i\in I^c} \ve_ix_{ij}\biggr|\, .\label{o2}
\end{align}
By Lemma \ref{app2} from the Appendix, 
\begin{align}\label{o3}
\ee'\biggl(\max_{1\le j\le p} \biggl|\sum_{i\in I^c} \ve_ix_{ij}\biggr|\biggr)\le \sigma \biggl(2\log (2p) \max_{1\le j\le p}\sum_{i\in I^c} x_{ij}^2\biggr)^{1/2}\le\sigma (2M^{1/2}n\log (2p))^{1/2}\, .
\end{align}
The desired result is obtained by combining \eqref{o1}, \eqref{o2} and \eqref{o3}, and taking unconditional expectation on both sides. 
\end{proof}

\begin{lmm}\label{mainlmm1}
\begin{align*}
\mathbb{E}\biggl(\sum_{i\in I}(x_i\beta^*-x_i\hat{\beta}^{(L,2)})^2\biggr)&\leq 3L^2(2Mn\log(2p^2))^{1/2} + 2L\sigma (2M^{1/2}n\log(2p))^{1/2}\,.
\end{align*}
\end{lmm}
\begin{proof}
Since $|\beta^*|_1$ and $|\bh^{(L,2)}|_1$ are both bounded by $L$, therefore
$$\sum_{i\in I}(x_i\beta^*-x_i\hat{\beta}^{(L,2)})^2 \leq \sum_{i\in I^c}(x_i\beta^*-x_i\hat{\beta}^{(L,2)})^2+\sup_{\beta\in \mathbb{R}^p\, :\, |\beta|_1\leq L}\varphi(\beta)\,.$$
Using Lemma \ref{l1} to bound the first term on the right, and Lemma \ref{l2} to bound the second, we get the desired result. 
\end{proof}
Next we bound the first term in \eqref{mainineq}.
\begin{lmm}\label{mainlmm2}
\begin{align*}
\mathbb{E}\biggl(\sum_{i \in I}\varepsilon_i(x_i\hat{\beta}^{(\hat{K}_1,2)}-x_i\hat{\beta}^{(L,2)})\biggr) &\leq (16\sigma^4n +8L^2M^{1/2}\sigma^2n)^{1/2} (\ee(\log (N_1+1)))^{1/2}\\
&\quad  + (n(n+5)\sigma^4 + n(n+1)\sigma^2L^2M^{1/2})^{1/2}\biggl(\frac{1+2^{-1/2}}{2}\biggr)^{n/2}\, . 
\end{align*}
\end{lmm}
\begin{proof}
For each $K$, let
$$W(K):=\sum_{i \in I}\varepsilon_ix_i\hat{\beta}^{(K,2)}\, .$$
Note that by conditional independence of $(\varepsilon_i)_{i\in I}$ and $\hat{\beta}^{(L,2)}$ given $I$, we know that the conditional distribution of $W(L)$ given $I$ and $(\varepsilon_i)_{i\in I^c}$ is normal with mean zero and variance $\sum_{i \in I}(x_i\hat{\beta}^{(L,2)})^2$. In particular,
$$\mathbb{E}\biggl(\sum_{i \in I}\varepsilon_ix_i\hat{\beta}^{(L,2)}\biggl)=\mathbb{E}\biggl(\ee''\biggl(\sum_{i \in I}\varepsilon_ix_i\hat{\beta}^{(L,2)}\biggr)\biggr) = 0\,.$$
(Note that this holds true even if $I$ or $I^c$ is empty.) Hence, to prove the lemma, it is enough to show that the right-hand side is an upper bound for the expectation of $W(\hat{K}_1)$. 

Let $\ck := \{0,\delta,2\delta,\ldots, N_1\delta\}$. Then 
\begin{align}\label{wk1}
W(\hat{K}_1) \leq \max_{K \in \mathcal{K}}W(K)\,.
\end{align}
We know that $\hat{K}_1$ minimizes $\sum_{i \in I}(y_i-x_i\hat{\beta}^{(K,2)})^2$ among all $K\in \mathcal{K}$. Therefore, in particular,
$$\sum_{i \in I}(y_i-x_i\hat{\beta}^{(\hat{K}_1,2)})^2 \leq \sum_{i \in I}(y_i-x_i\hat{\beta}^{(0,2)})^2=\sum_{i \in I}y^2_i$$
since $\hat{\beta}^{(0,2)}=0$. This implies that
\begin{align}\label{x1}
\biggl(\sum_{i \in I}(x_i\hat{\beta}^{(\hat{K}_1,2)})^2\biggr)^{1/2} \leq \biggl( \sum_{i \in I}y_i^2\biggr)^{1/2}+\biggl(\sum_{i \in I}(y_i-x_i\hat{\beta}^{(\hat{K}_1,2)})^2 \biggr)^{1/2} \leq 2\biggl( \sum_{i \in I}y_i^2\biggr)^{1/2}\, .
\end{align}
Consequently,
\begin{align}\label{x2}
\biggl|\sum_{i \in I}\varepsilon_ix_i\hat{\beta}^{(\hat{K}_1,2)}\biggr|&\leq \biggl(\sum_{i \in I}\varepsilon_i^2	\sum_{i \in I}(x_i\hat{\beta}^{(\hat{K}_1,2)})^2\biggr)^{1/2} \leq 2\biggl(\sum_{i \in I}\varepsilon_i^2	\sum_{i \in I}y_i^2\biggr)^{1/2}\,.
\end{align}
Let $m:=16|I|\sigma^2+8\sum_{i\in I}(x_i\beta^*)^2$, and 
$$\mathcal{K}':=\biggl\{K\in \mathcal{K}:\sum_{i\in I}(x_i\hat{\beta}^{(K,2)})^2\leq m\biggr\}\, .$$
Then by \eqref{wk1}, at least one of the two following inequalities must hold:
\begin{align*}
&W(\hat{K}_1) \leq \max_{K \in \mathcal{K}'}W(K), \\
&\sum_{i\in I} (x_i \bh^{(\hk_1,2)})^2 > m\,.
\end{align*}
In the latter case, \eqref{x1} implies that
\[
4\sum_{i\in I} y_i^2 > m\,,
\]
and \eqref{x2} implies that
\[
W(\hk_1) \le 2\biggl(\sum_{i\in I} \ve_i^2\sum_{i\in I} y_i^2\biggr)^{1/2}\,.
\]
Thus, 
\begin{align}\label{wk2}
W(\hk_1) &\le \max_{K \in \mathcal{K}'}W(K) + 2\biggl(\sum_{i\in I} \ve_i^2\sum_{i\in I} y_i^2\biggr)^{1/2}1_{\{4\sum_{i\in I} y_i^2 > m\}}\,,
\end{align}
where $1_A$ is our notation for the indicator of an event $A$.

If we condition on $I$ and $(\varepsilon_i)_{i\in I^c}$ then $m$ and $(\hat{\beta}^{(K,2)})_{K\in \mathcal{K}}$ become non-random but $(\varepsilon_i)_{i \in I}$ are still i.i.d. $N(0,\sigma^2)$ random variables. Thus, for each $K \in \mathcal{K}'$, $W(K)$ is conditionally a Gaussian random variable with mean zero and variance bounded by $m\sigma^2$. Therefore by Lemma \ref{app2} of the Appendix, 
\begin{equation*}
\mathbb{E}''\bigl(\max_{K \in \mathcal{K}'}W(K)\bigr) \leq \sigma \sqrt{2m\log |\mathcal{K}'|}\leq \sigma \sqrt{2m\log |\mathcal{K}|} \, .
\end{equation*}
Taking unconditional expectation gives 
\begin{align}
\mathbb{E}\bigl(\max_{K \in \mathcal{K}'}W(K)\bigr)&\leq \sigma\sqrt{2}\ee(\sqrt{m\log |\mk|})\nonumber\\
&\le \sigma \sqrt{2}(\ee(m) \ee(\log |\ck|))^{1/2}\,.\label{wk3}
\end{align}
Note that
\begin{align}
\ee(m) &= 8n\sigma^2 + 4\sum_{i=1}^n (x_i \beta^*)^2\nonumber \\
&= 8n\sigma^2 + 4\sum_{i=1}^n \sum_{1\le j,k\le p} x_{ij}x_{ik}\beta^*_j \beta^*_k\nonumber\\
&\le 8n\sigma^2 + 4|\beta^*|_1^2 \max_{1\le j,k\le p} \biggl|\sum_{i=1}^n x_{ij}x_{ik}\biggr|\nonumber\\
&\le 8n\sigma^2 + 4L^2 \max_{1\le j\le p} \sum_{i=1}^n x_{ij}^2\nonumber\\
&\le 8n\sigma^2 + 4L^2M^{1/2}n \,.\label{mbd}
\end{align}
On the other hand,
\[
\log |\ck|=\log(N_1+1)\,.
\]
Therefore by \eqref{wk3},
\begin{align}\label{ww1}
\mathbb{E}\bigl(\max_{K \in \mathcal{K}'}W(K)\bigr) &\leq  (16\sigma^4n +8L^2M^{1/2}\sigma^2n)^{1/2} (\ee(\log (N_1+1)))^{1/2}\,.
\end{align}
Now we are left to control the second term on the right hand side of \eqref{wk2}. If we call it $S$ then
\begin{align}\label{w0}
(\mathbb{E}(S))^2&\leq 4\,\mathbb{E}\biggl(\sum_{i \in I}\varepsilon_i^2	\sum_{i \in I}y_i^2\biggr)\mathbb{P}\biggl(4\sum_{i \in I}y_i^2> m\biggr)\, .
\end{align}
Let 
\[
\xi_i :=
\begin{cases}
1&\text{ if } i\in I\, ,\\
0&\text{ if } i\in I^c\, .
\end{cases}
\]
Then 
\begin{align}
\mathbb{E}\biggl(\sum_{i \in I}\varepsilon_i^2	\sum_{i \in I}y_i^2\biggr) &= \sum_{i,j=1}^n \ee(\xi_i \xi_j \ve_i^2y_j^2)\nonumber\\
&= \frac{1}{2}\sum_{i=1}^n \ee(\ve_i^2 y_i^2) + \frac{1}{4} \sum_{1\le i\ne j\le n} \ee(\ve_i^2)\ee(y_j^2)\nonumber\\
&= \frac{1}{2}\sum_{i=1}^n (3\sigma^4 + (x_i\beta^*)^2\sigma^2) + \frac{1}{4} \sum_{1\le i\ne j\le n} (\sigma^4 + (x_j\beta^*)^2\sigma^2)\nonumber\\
&= \frac{n(n+5)\sigma^4}{4} + \frac{(n+1)\sigma^2}{4}\sum_{i=1}^n (x_i\beta^*)^2\, .\label{last1}
\end{align}
Next, note that by Chebychev's inequality,
\begin{align}
\pp'\biggl(4\sum_{i\in I} y_i^2 > m\biggr) &\le \ee'\exp\biggl(\frac{4\sum_{i\in I} y_i^2 - m}{16\sigma^2}\biggr)\nonumber \\
&\le e^{-m/16\sigma^2} \prod_{i\in I}\ee'(e^{y_i^2/4\sigma^2})\, .\label{p1}
\end{align}
By Lemma \ref{app4} from the Appendix,
\begin{align*}
\ee'(e^{y_i^2/4\sigma^2}) &= \sqrt{2} e^{(x_i\beta^*)^2/2\sigma^2}\, .
\end{align*}
Thus,
\begin{align*}
e^{-m/16\sigma^2} \prod_{i\in I}\ee'(e^{y_i^2/4\sigma^2}) &= 2^{|I|/2} e^{-|I|}\le 2^{-|I|/2}\, .
\end{align*}
Therefore by \eqref{p1}, 
\begin{align}\label{last2}
\pp\biggl(4\sum_{i\in I} y_i^2 > m\biggr) &\le \ee(2^{-|I|/2}) = \prod_{i=1}^n \ee(2^{-\xi_i/2}) = \biggl(\frac{1+2^{-1/2}}{2}\biggr)^n\, .
\end{align}
Combining \eqref{w0}, \eqref{last1} and \eqref{last2}, we get
\begin{align}
(\ee(S) )^2 &\le (n(n+5)\sigma^4 + n(n+1)\sigma^2L^2M^{1/2})\biggl(\frac{1+2^{-1/2}}{2}\biggr)^n\, .\label{ww2}
\end{align}
The proof is completed by combining \eqref{wk1}, \eqref{ww1} and \eqref{ww2}.
\end{proof}

We are now ready to complete the proof of Theorem \ref{mainthm}.

\begin{proof}[Proof of Theorem \ref{mainthm}]
Combine inequality \eqref{mainineq} with Lemma \ref{mainlmm1} and Lemma \ref{mainlmm2} to get a bound for $\ee(\sum_{i\in I} (\hm_i-\mu_i^*)^2;\, N_1\delta\ge |\beta^*|_1)$, and add a similar term for $I^c$. This shows that 
\begin{align}
&\ee(\|\hm-\mu^*\|^2;\, N_1\delta \ge |\beta^*|_1,\, N_2\delta \ge |\beta^*|_1) \nonumber\\
&\le 6L^2(2Mn\log(2p^2))^{1/2} + 4L\sigma (2M^{1/2}n\log(2p))^{1/2}\nonumber\\
&\quad + (16\sigma^4n +8L^2M^{1/2}\sigma^2n)^{1/2} (\sqrt{l_1}+\sqrt{l_2})\nonumber\\
&\quad  + 2(n(n+5)\sigma^4 + n(n+1)\sigma^2L^2M^{1/2})^{1/2}\biggl(\frac{1+2^{-1/2}}{2}\biggr)^{n/2}\, .\label{c1}
\end{align}
 Now note that by the definition of $\hk$,
\begin{align*}
\|\mu^*-X\bh^{(\hk)}\| &\le \|\mu^*-\hm\| + \|\hm - X\bh^{(\hk)}\|\\
&\le \|\mu^*-\hm\|+\|\hm - X\bh^{(L)}\|\\
&\le 2\|\hm - \mu^*\| + \|\mu^*-X\bh^{(L)}\|\,.
\end{align*}
Thus,
\begin{align}\label{c2}
\|\mu^*-X\bh^{(\hk)}\|^2 &\le 8\|\hm - \mu^*\|^2 + 2\|\mu^*-X\bh^{(L)}\|^2\,.
\end{align}
Next observe that by the definition of $\bh^{(L)}$ and the fact that $|\beta^*|_1\le L$,
\[
\|Y-X\bh^{(L)}\|^2 \le \|Y-X\beta^*\|^2 = \|\ve\|^2\, .
\]
Since $|\beta^*|_1$ and $|\bh^{(L)}|_1$ are both bounded by $L$, the above inequality implies that  
\begin{align}
\|\mu^*-X\bh^{(L)}\|^2 &\le 2\sum_{i=1}^n \ve_i(x_i \bh^{(L)} - x_i \beta^*)\nonumber\\
&\le 2\sum_{i=1}^n \sum_{j=1}^p \ve_i x_{ij}(\bh^{(L)}_j - \beta^*_j)\nonumber\\
&\le 4L \max_{1\le j\le p} \biggl|\sum_{i=1}^n \ve_i x_{ij}\biggr|\, .\label{c3}
\end{align}
By Lemma \ref{app2}, 
\begin{align}
\ee\biggl(\max_{1\le j\le p}\biggl|\sum_{i=1}^n \ve_i x_{ij}\biggr|\biggr) &\le \sigma\sqrt{2\log (2p)} \max_{1\le j\le p} \biggl(\sum_{i=1}^n x_{ij}^2\biggr)^{1/2}\nonumber\\
&\le \sigma\bigl(2n M^{1/2}\log (2p)\bigr)^{1/2}\, .\label{c4}
\end{align}
Combining \eqref{c1}, \eqref{c2}, \eqref{c3} and \eqref{c4} completes the proof of Theorem \ref{mainthm}. 
\end{proof}

\section{Proofs of Theorem \ref{sigma} and Corollary \ref{sigmacor}}\label{corsec}

\begin{proof}[Proof of Theorem \ref{sigma}]
Let $\hmu := X\bh^{\textup{CV}}$. Note that 
\begin{align*}
|\hat{\sigma}^2 - \sigma^2|&= \biggl|\frac{\|Y-\hmu\|^2-n\sigma^2}{n}\biggr|\\
&= \biggl|\frac{\|\ve\|^2 + 2\ve \cdot (\mu^*-\hmu) + \|\mu^*-\hmu\|^2-n\sigma^2}{n}\biggr|\\
&\le \frac{|\|\ve\|^2 - n\sigma^2|}{n} + \frac{2\|\ve\|\|\mu^*-\hmu\|}{n} + \frac{\|\mu^*-\hmu\|^2}{n}\, .
\end{align*}
Let $A$ be the event that $N_1\delta$ and $N_2\delta$ both exceed $|\beta^*|_1$. Since $\ee(\ve_i^2)=\sigma^2$ and $\var(\ve_i^2) = 2\sigma^4$, and the $\ve_i$'s are independent, therefore
\begin{align*}
\ee\biggl(\frac{|\|\ve\|^2 - n\sigma^2|}{n}; A\biggr)&\le \ee\biggl(\frac{|\|\ve\|^2 - n\sigma^2|}{n}\biggr)\\
&\le \biggl(\var\biggl(\frac{\|\ve\|^2}{n}\biggr)\biggr)^{1/2}= \sigma^2\sqrt{\frac{2}{n}}\, .
\end{align*}
Next, note that by the Cauchy-Schwarz inequality,
\begin{align*}
\ee\biggl(\frac{2\|\ve\|\|\mu^*-\hmu\|}{n}; A\biggr)&\le \frac{2}{n}\bigl(\ee(\|\ve\|^2) \ee(\|\mu^*-\hmu\|^2; A)\bigr)^{1/2}\\
&= 2\sigma \biggl(\ee\biggl(\frac{\|\mu^*-\hmu\|^2}{n}; A\biggr)\biggr)^{1/2}\, .
\end{align*}
An application of Theorem \ref{mainthm} completes the proof.
\end{proof}
\begin{proof}[Proof of Corollary \ref{sigmacor}]
This is simply a question of verifying that the error bound tends to zero as $n\ra\infty$. By assumptions (i) and (ii), the quantities $C_1$ and $C_2$ remain bounded and $E_n$ tends to zero, and the events $N_1\delta\ge |\beta^*|_1$ and $N_2\delta\ge |\beta^*|_1$ have probability tending to one as $n\ra\infty$. Finally by assumption (iii), $\sqrt{l_1/n}$, $\sqrt{l_2/n}$ and $\sqrt{\log(2p)/n}$ all tend to zero. To complete the proof, let $A$ be the event that both $N_1\delta$ and $N_2\delta$ exceed $|\beta^*|_1$, and note that for any $s >0$,
\begin{align*}
\pp(|\hat{\sigma}^2-\sigma^2|\ge s) &\le \pp(\{|\hat{\sigma}^2-\sigma^2|\ge s\}\cap A) + \pp(A^c)\\
&\le \frac{1}{s}\ee(|\hat{\sigma}^2-\sigma^2|;\, A) + \pp(A^c)\,.
\end{align*}
We have argued above that both terms in the last line tend to zero. This completes the proof.
\end{proof}

\appendix
\setcounter{section}{1}
\setcounter{thm}{0}
\setcounter{equation}{0}
\section*{Appendix}
This appendix contains a few simple lemmas that have been used several times in the proof of Theorem \ref{mainthm}. All of these are well-known results. We give proofs for the sake of completeness. 
\begin{lmm}\label{app1}
Suppose that $\xi_1,\dots,\xi_m$ are independent, mean zero random variables, and $\gamma_1, \ldots, \gamma_m$ are constants such that $|\xi_i|\leq \gamma_i$ almost surely for each $i$. Then for each $\theta \in \mathbb{R}$,
$$\mathbb{E}\left(e^{\theta\sum_{i=1}^m\xi_i}\right)\leq e^{\theta^2\sum_{i=1}^{m}\gamma_i^2/2}$$
\end{lmm}
\begin{proof}
By independence, it suffices to prove the lemma for $m=1$. Also, without loss of generality, we may take $\theta=1$. Accordingly, let $\xi$ be a random variable and $\gamma$ be a constant such that $|\xi|\le \gamma$ almost surely. Let
\[
\alpha := \frac{1}{2}\biggl(1-\frac{\xi}{\gamma}\biggr)\,.
\]
Since $|\xi|\le \gamma$, therefore $\alpha\in [0,1]$. Thus, by the convexity of the exponential map, 
\[
e^\xi = e^{-\gamma\alpha +\gamma(1-\alpha)} \le \alpha e^{-\gamma} + (1-\alpha)e^\gamma\,.
\]
Taking expectation on both sides, and using the assumption that $\ee(\xi)=0$, we get
\[
\ee(e^\xi) \le \cosh \gamma\, .
\]
It is easy to verify that $\cosh \gamma\le e^{\gamma^2/2}$ by power series expansion.
\end{proof}
\begin{lmm}\label{app2}
Suppose that $\xi_1,\dots, \xi_m$ are mean zero random variables, and $\sigma$ is a constant such that $\mathbb{E}(e^{\theta \xi_i})\leq e^{\theta^2\sigma^2/2}$ for each $\theta\in \mathbb{R}$. Then
$$\mathbb{E}(\max_{1\leq i\leq m}\xi_i) \leq \sigma\sqrt{2\log m}$$
and
$$\mathbb{E}\bigl(\max_{1\leq i \leq m}|\xi_i|\bigr)\leq \sigma\sqrt{2\log(2m)}\,.$$
\end{lmm}
\begin{proof}
Note that for any $\theta \ge 0$,
\begin{align*}
\ee(\max_{1\le i\le m}\xi_i) &= \frac{1}{\theta}\ee(\log e^{\theta\max_{1\le i\le m} \xi_i})\\
&\le \frac{1}{\theta}\ee\biggl(\log \sum_{i=1}^m e^{\theta \xi_i}\biggr)\\
&\le \frac{1}{\theta} \log \sum_{i=1}^m \ee(e^{\theta \xi_i})\\
&\le \frac{\log m}{\theta} + \frac{\theta \sigma^2}{2}\, .
\end{align*}
The proof of the first inequality is completed by choosing $\theta = \sigma^{-1}\sqrt{2\log m}$. The second inequality is proved by applying the first inequality to the collection $\xi_1,\ldots, \xi_m, -\xi_1,\ldots,-\xi_m$.
\end{proof}
\begin{lmm}\label{app4}
If $Z\sim N(\mu, \sigma^2)$, then for any $a>1$, 
\[
\ee(e^{Z^2/2a\sigma^2}) = e^{\mu^2/2(a-1)\sigma^2}\sqrt{\frac{a}{a-1}}\, .
\]
\end{lmm}
\begin{proof}
This is a simple Gaussian computation. Just note that since 
\[
\int_{-\infty}^\infty e^{-(x-\alpha)^2/2\beta^2} dx = \sqrt{2\pi \beta^2}
\]
for any $\alpha$ and $\beta$, therefore
\begin{align*}
\int_{-\infty}^\infty e^{x^2/2a\sigma^2} e^{-(x-\mu)^2/2\sigma^2} dx &= \int_{-\infty}^\infty e^{(x^2 - a(x-\mu)^2)/2a\sigma^2}dx\\
 &= \int_{-\infty}^\infty e^{-(a-1)(x-\mu a/(a-1))^2/2a\sigma^2 + \mu^2/2(a-1)\sigma^2} dx\\
 &= e^{\mu^2/2(a-1)\sigma^2} \sqrt{\frac{2\pi a\sigma^2}{a-1}}\, .
\end{align*}
Now divide on both sides by $\sqrt{2\pi \sigma^2}$ to complete the proof. 
\end{proof}

\vskip.2in
\noindent{\bf Acknowledgments.} We thank Stephen Reid and Rob Tibshirani for many helpful discussions and comments.
\vskip.2in 


\end{document}